\documentclass[12pt]{amsart}
\usepackage{geometry}
\usepackage{amsmath,amssymb,amsthm,mathtools}
\usepackage{enumitem}
\usepackage[backref,colorlinks=true,linkcolor=blue,citecolor=blue,urlcolor=blue]{hyperref}

\setlist[itemize]{leftmargin=2em}
\setlist[enumerate]{leftmargin=2em}

\numberwithin{equation}{section}
\numberwithin{figure}{section}

\newtheorem{theorem}{Theorem}[section]
\newtheorem{lemma}[theorem]{Lemma}

\newtheorem{corollary}[theorem]{Corollary}
\theoremstyle{definition}
\newtheorem{definition}[theorem]{Definition}
\theoremstyle{remark}
\newtheorem{remark}[theorem]{Remark}

\newcommand{\R}{\mathbb{R}}
\newcommand{\Lup}{L^{\operatorname{up}}}
\newcommand{\Ldown}{L^{\operatorname{down}}}
\newcommand{\im}{\operatorname{im}}
\newcommand{\calH}{\mathcal{H}}
\newcommand{\calF}{\mathcal{F}}
\newcommand{\tH}{\widetilde H}
\newcommand{\p}{\partial}
\def \d{\delta}

\begin{document}

\title[Upper Bounds for Largest Laplacian Eigenvalue]{Upper Bounds for the Largest Laplacian Eigenvalue of Simplicial Complexes}
\author [H.-Z. Zhang]{Huan-Zhi Zhang}
\address{\small School of Mathematical Sciences, Anhui University, Hefei 230601, P. R. China}
\email{zhanghz@stu.ahu.edu.cn}
\author [Y.-Z. Fan]{Yi-Zheng Fan*}
\address{\small Center for Pure Mathematics, School of Mathematical Sciences, Anhui University, Hefei 230601, P. R. China}
\email{fanyz@ahu.edu.cn}
\thanks{*Corresponding author. Supported by National Natural Science Foundation of China (No. 12331012).}
	
	

\begin{abstract}
Let $K$ be a finite $r$-dimensional simplicial complex with vertex set $V$ of size $n$.
We study the largest eigenvalue of the combinatorial  $(r-1)$-up Laplacian $\Lup_{r-1}(K)$.
It is known that
\[
\lambda_{\max}\bigl(\Lup_{r-1}(K)\bigr)\le n.
\]
We first give a homological equality criterion for this universal bound,
namely, the equality holds if and only if the $r$-dimensional complement $K^c$ of $K$ has a nonzero reduced homology $\tH_{r-1}(K^c,\R)$.
For $r=1$, this is the classical graph condition that the complement graph is disconnected.

Secondly, we prove a sharper upper bound for $\lambda_{\max}(\Lup_{r-1}(K))$:
\[ \lambda_{\max}(\Lup_{r-1}(K)) \le \max_{F\in S_r(K)} \bigl|\bigcup_{E \in \p F} N_K(E) \bigr| \le n,\]
where, for an $(r-1)$-face $E$, $N_K(E)$ denotes the set of vertices $u$ outside $E$ such that the union $E \cup \{u\}$ is an $r$-face of $K$.
This is the high-dimensional analog of the graph Laplacian bound.
We give an explicit characterization of the equality case, and construct a broad family attaining the bound, namely, the partite semiregular complexes with admissible additions.
\end{abstract}

\maketitle

\section{Introduction}
The largest Laplacian eigenvalue is one of the basic spectral parameters of a finite graph.
Let $G$ be a graph on $n$ vertices, and let $L(G)$ be its Laplacian matrix.
Anderson and Morley \cite{AM1985} proved that every Laplacian eigenvalue is at most $n$, with equality if and only if the complement graph $G^c$ is disconnected; see also \cite{M1994, M1998}.
They also obtained the edge-degree estimate
\begin{equation}\label{AD-bound}
    \lambda_{\max}(L(G))  \le \max_{\{u,v\}\in E(G)}(\deg u+\deg v),
\end{equation}
  and characterized the equality case when $G$ is connected.
Subsequent work refined these estimates by replacing bounds expressed only in terms of vertex degrees with bounds that incorporate the local overlap structure of neighborhoods.
Rojo, Soto and Rojo \cite{RSR2000} provided an always nontrivial
bound involving common neighbors, and Das \cite{Das2003} sharpened this to
\begin{equation}\label{Das-bound}
   \lambda_{\max}(L(G))  \le   \max_{\{u,v\}\in E(G)} |N_G(u)\cup N_G(v)|,
\end{equation}
where $N_G(w)$ denotes the neighborhood of a vertex $w$ in the graph $G$.
Since
$$
|N_G(u)\cup N_G(v)|=\deg u+ \deg v-|N_G(u)\cap N_G(v)|
$$
for an edge $\{u,v\}$, this bound is never larger than the Anderson-Morley bound \eqref{AD-bound} and is always at most $n$.
The equality case for Das' bound \eqref{Das-bound} was studied by Yu, Lu, and Tian \cite{YLT2005}, who showed that the extremal connected graphs are built from semiregular bipartite graphs with controlled additions inside the parts.
Das also characterized equality for several related Laplacian upper bounds \cite{Das2004}.

There is a parallel higher-dimensional theory for simplicial complexes.
An \emph{abstract simplicial complex}, or simply a \emph{complex}, is a collection $K$ of finite sets closed under inclusion.
The graph Laplacian is the first combinatorial up Laplacian of a simplicial complex: if a graph is regarded as a one-dimensional complex, then $L(G)=\partial_1\partial_1^*$ on $0$-chains.
For a simplicial complex $K$, the combinatorial $(r-1)$-up Laplacian is
\[  \Lup_{r-1}(K)=\partial_r\partial_r^*,\]
acting on $(r-1)$-chains.
Combinatorial Laplacians of this type go back to Eckmann's discrete Hodge theory \cite{Eckmann1944} and have been developed in many directions.
Duval and Reiner \cite{DR2002} studied their spectra for shifted complexes and proved, among
other things, the universal estimate
\begin{equation}\label{DR-bound}
\lambda_{\max}(\Lup_{r-1}(K))\le n
\end{equation}
for a complex on $n$ vertices.
Horak and Jost \cite{HJ2013a,HJ2013b} provided a systematic framework for discrete
Laplace operators on weighted simplicial complexes and established general eigenvalue bounds and interlacing inequalities.
More recently, Fan, Wu, and Wang \cite{FWW2025} studied the largest up-Laplacian eigenvalue through the associated signless Laplacian and incidence signed graph, obtaining a high-dimensional analog of Anderson–Morley upper bounds
\begin{equation}\label{Fan-bound}
\lambda_{\max}(\Lup_{r-1}(K))\le \max_{F \in S_r(K)} \sum_{E \in \p F} \deg E,
\end{equation}
and the balancedness criteria for equality.

The aim of the present paper is to generalize these results on graphs to simplicial complexes.
The first main result is a homological equality criterion for the universal bound $n$.
Given an $r$-complex $K$ on vertex set $V$ of size $n$ with $r\ge 1$, define the $r$-dimensional complement $K^c$ to have the complete $(r-1)$-skeleton on $V$ and exactly those $r$-faces that are missing from $K$.
We prove
\[
\lambda_{\max}(\Lup_{r-1}(K))\le n
\]
and equality holds if and only if the reduced homology $\widetilde H_{r-1}(K^c,\R)\ne 0$.
For $r=1$, this is exactly the classical graph statement that $\lambda_{\max}(L(G))=n$ if and only if the complement graph $G^c$ is disconnected.

The second main result is a Das-type upper bound for simplicial complexes.
For an $(r-1)$-face $E$, let $N_K(E)$ be the set of vertices $u$ outside $E$ such that $E \cup \{u\}$ is an $r$-face of $K$.
The \emph{degree} of $E$, denoted by $\deg E$, is the cardinality of $N_K(E)$.
For an $r$-face $F$, we call the cardinality $|\cup_{E \in \p F} N_K(E) |$ the \emph{local number} of $F$, denoted by $\ell_K(F)$.
Then
\begin{equation}\label{New-bound}
\lambda_{\max}(\Lup_{r-1}(K)) \le \max_{F\in S_r(K)} \bigl|\bigcup_{E \in \p F} N_K(E) \bigr|
\le n.
\end{equation}
Since
$$
\bigl|\bigcup_{E \in \p F} N_K(E) \bigr| \le \sum_{E \in \p F} |N_K(E)|
=\sum_{E \in \p F}\deg E,
$$
the above bound is never larger than the bound in \eqref{Fan-bound}.
When $r=1$, this becomes precisely Das' graph bound \eqref{Das-bound}.
The proof is based on a decomposition of the down Laplacian into local matrices supported on $(r+2)$-vertex sets.
This decomposition also gives an exact quadratic form for the down Laplacian and an equality criterion: equality in the upper bound holds if and only if a certain mapping has a nonzero kernel that contains a vector supported on $r$-faces with maximum local number.

The paper is organized as follows.
Section \ref{section_basic} introduces notation and conventions for
simplicial complexes and combinatorial Laplacians.
Section \ref{section_n} proves the universal $n$ bound and the homological equality condition.
Section \ref{section_new} proves a sharper upper bound, recovering Das' theorem in the graph case;
and derives the equality criterion for the sharper upper bound.
Section \ref{section_fam} constructs a family of partite semiregular complexes with admissible additions that attain the sharper upper bound.
In dimension one, this family specializes to the semiregular bipartite graph families appearing in the equality characterization for Das' bound.

\section{Preliminaries}\label{section_basic}
Let $K$ be a finite simplicial complex with vertex set $V=V(K)$.
An $r$-face of $K$ is an element of $K$ with cardinality $r+1$.
The set of $r$-faces of $K$ is denoted by $S_r(K)$.
The \emph{dimension} of an $r$-face is $r$, and the \emph{dimension} of $K$ is the maximum dimension of all faces of $K$.
A complex of dimension $r$ is usually called an \emph{$r$-complex}.
We take the empty set as an element of $K$ whose dimension is $-1$.
For an integer $p\ge -1$, the \emph{$p$-skeleton }of $K$ is denoted by $K^{(p)}$,
that is, the subcomplex consisting of all faces of $K$ of dimension at most $p$.

For a finite vertex set $V$, let $\Delta_V$ denote the full simplex on $V$, which can be considered a simplicial complex  with all possible subsets of $V$ as faces.
If $|V|=n$, we also write $\Delta_n$ for a simplex on $n$ vertices.
Thus $\Delta_n$ has dimension $n-1$.

For two $i$-faces of $K$ that share an $(i-1)$-face, we refer to them as \emph{down neighbors}.
An \emph{$i$-down path} is a sequence of $i$-faces $F_1,F_2,\ldots,F_m$ such that $F_j$ and $F_{j+1}$ are down neighbors for all $j$.
The complex $K$ is called \emph{$i$-down path connected} if any two $i$-faces are connected by an $i$-down path.

An \emph{ordered finite set} is a finite set together with a linear order on its elements.
Let $\tau\subseteq \sigma$ be two ordered finite sets.
We define $(\sigma:\tau)$ to be the sign of the permutation that sends the ordered set $\sigma$ to the ordered set $(\sigma\setminus\tau,\tau)$, where the order on $\sigma\setminus\tau$ is induced by the order on $\sigma$.
Usually, we write $[\sigma]$ for an ordered set $\sigma$.

In particular, a face $F$  is \emph{oriented} if we assign an order to its vertices and write $[F]$.
Two orderings of the vertices are said to determine the same orientation if there is an even permutation transforming one ordering into the other.
If the permutation is odd, then the orientations are opposite.
After choosing one orientation for each $r$-face, the \emph{$r$-th chain group} $C_r(K,\R)$ of $K$ with coefficients in $\R$ is the vector space over $\R$ with basis
$B_r(K)=\{[F]:F\in S_r(K)\}$.
The $r$-th boundary map $\partial_r:C_r(K,\R)\to C_{r-1}(K,\R)$ is defined by
\[
\partial_r[v_0,\ldots,v_r]=\sum_{j=0}^r(-1)^j[v_0,\ldots,\widehat v_j,\ldots,v_r],
\]
where $\widehat v_j$ indicates that the vertex $v_j$ is omitted.
For an $r$-face $F$, we write $\partial F$ for the set of all $(r-1)$-faces contained in $F$.

Throughout the paper, we use the augmented chain complex.
Set $C_{-1}(K,\R)=\R$, and define the augmentation map $\partial_0:C_0(K,\R)\to \R$ by
$\partial_0(\sum_v a_v v)=\sum_v a_v$.
These boundary maps satisfy $\partial_r\partial_{r+1}=0$, and hence they form the augmented chain complex
\[
\cdots \xrightarrow{\partial_{r+1}} C_r(K,\R) \xrightarrow{\partial_r} C_{r-1}(K,\R)
\to\cdots \xrightarrow{\partial_1} C_0(K,\R) \xrightarrow{\partial_0} C_{-1}(K,\R) \to 0.
\]
The kernel $\ker\partial_r$ is called the \emph{space of $r$-cycles}.
The image $\im\partial_{r+1}$ is called the \emph{space of $r$-boundaries}.
Since $\partial_r\partial_{r+1}=0$, we have $\im\partial_{r+1}\subseteq \ker\partial_r$.
The quotient
\[
\widetilde H_r(K,\R)=\ker\partial_r/\im\partial_{r+1}
\]
is called the \emph{$r$-th reduced homology group} of $K$ over $\R$.

For each $r$, endow $C_r(K,\R)$ with a positive definite inner product $\langle \cdot, \cdot \rangle$ that makes the basis $B_r(K)$ orthonormal.
The adjoint $\p_r^*: C_{r-1}(K,\R) \to C_{r}(K,\R)$ of $\p_r$  is defined by
\[
\langle \p_r f_1, f_2 \rangle_{C_{r-1}(K,\R)} = \langle f_1, \p_r^* f_2\rangle_{C_r(K,\R)}
\]
for all $f_1 \in C_r(K,\mathbb{R})$, $f_2 \in C_{r-1}(K,\mathbb{R})$.
The \emph{combinatorial $r$-up Laplacian} and \emph{combinatorial $r$-down Laplacian} are
\[
\Lup_{r}(K)=\p_{r+1}\p_{r+1}^*, \quad \Ldown_{r}(K)=\p_{r}^*\p_{r}.
\]

The matrices $\p_r\p_r^*$ and $\p_r^*\p_r$ have the same nonzero eigenvalues. Hence,
\begin{equation}\label{eq-ud}
\lambda_{\max}\bigl(\Lup_{r-1}(K)\bigr)        =\lambda_{\max}\bigl(\Ldown_r(K)\bigr)
\end{equation}
whenever $S_r(K)\ne\emptyset$.
We shall frequently use \eqref{eq-ud} to pass from
$\Lup_{r-1}(K)$ to $\Ldown_r(K)$. This is useful because the matrix
$\Ldown_r(K)$ is indexed by $r$-faces, and its off-diagonal entries are
controlled by intersections of pairs of $r$-faces.

\section{Equality at the universal bound \texorpdfstring{$n$}{n}}\label{section_n}

\begin{definition}
Let $K$ be an $r$-dimensional complex on vertex set $V$, where  $r\ge 1$.
The \emph{$r$-dimensional complement} $K^c$ of $K$ is the complex
\[
K^c=\Delta_V^{(r-1)}\cup \{F\subseteq V: |F|=r+1,\ F\notin S_r(K)\}.
\]
Equivalently,
\[
 S_r(K^c)=\binom{V}{r+1}\setminus S_r(K),
\]
and $K^c$ has the complete $(r-1)$-skeleton on $V$.
\end{definition}

To avoid confusion, we use $\p_r^K$ to stress the boundary operator acting on the complex $K$.

\begin{lemma}\label{lem:full-simplex}
Let $\Delta_V$ be the full simplex on $n$ vertices, and let $\Delta_V^{(r)}$ be its $r$-skeleton.
Then
\begin{equation}\label{eq:full-total}
 \Lup_{r-1}(\Delta_V^{(r)}) +\Ldown_{r-1}(\Delta_V^{(r)}) = nI.
\end{equation}
Moreover, $\Lup_{r-1}(\Delta_V^{(r)})$
has only the eigenvalues $0$ and $n$, and its $n$-eigenspace is
\[
 \im \p_r^{\Delta}=\ker \p_{r-1}^{\Delta}.
\]
\end{lemma}

\begin{proof}
The \eqref{eq:full-total} be verified directly from the definition.
Since the full simplex is reduced-acyclic, the Hodge decomposition gives
\[
C_{r-1}(\Delta_V,\R)=\im \p_r^{\Delta}\oplus \im(\p_{r-1}^{\Delta})^*;
\]
or see \cite[Theorem 3.3]{DR2002}.
On $\im\p_r^\Delta$, the down Laplacian $\Ldown_{r-1}(\Delta_V^{(r)})$ vanishes; hence, \eqref{eq:full-total} gives $\Lup_{r-1}(\Delta_V^{(r)})|_{\im\p_r^\Delta}=nI$.
On $\im(\p_{r-1}^\Delta)^*$,  the up Laplacian $\Lup_{r-1}(\Delta_V^{(r)})$ vanishes, namely,  $\Lup_{r-1}(\Delta_V^{(r)})|_{\im(\p_{r-1}^\Delta)^*}=0$.
Therefore, $\Lup_{r-1}(\Delta_V^{(r)})$ has a spectrum consisting of $0$ and $n$, and the eigenspace associated with $n$ is $\im\p_r^\Delta$.
Acyclicity of $\Delta_V$ gives $\im\p_r^\Delta=\ker\p_{r-1}^\Delta$.
\end{proof}

\begin{theorem}\label{eq-n}
Let $K$ be an $r$-dimensional finite simplicial complex on vertex set $V$,
where $|V|=n$ and $r\ge 1$.
Then
\[
\lambda_{\max}\bigl(\Lup_{r-1}(K)\bigr)\le n.
\]
Moreover, equality holds if and only if $\widetilde H_{r-1}(K^c,\R)\ne 0$.
\end{theorem}

\begin{proof}
Let $\Delta:=\Delta_V$ be the full simplex on $V$, and let $L_{\Delta}:=\Lup_{r-1}(\Delta^{(r)})$.
We regard all $(r-1)$-up Laplacians as acting on $C_{r-1}(\Delta,\R)$ by adding zero rows and columns corresponding to missing $(r-1)$-faces.
This extension only adds zero eigenvalues.
Since the $r$-faces of $K$ and $K^c$ partition $\binom{V}{r+1}$, we have
\begin{equation}\label{total}
  L_\Delta=\Lup_{r-1}(K)+\Lup_{r-1}(K^c).
\end{equation}
Both summands on the right are positive semidefinite, and the largest eigenvalue of $L_\Delta$ is $n$.
Thus, for every $x\in C_{r-1}(\Delta,\R)$,
\[
\langle \Lup_{r-1}(K)x,x\rangle  \le \langle L_\Delta x,x\rangle  \le n\|x\|^2.
\]
Hence $\lambda_{\max}(\Lup_{r-1}(K))\le n$.

Assume first that $\lambda_{\max}(\Lup_{r-1}(K))=n$.
Choose a unit vector $x$ such that $\Lup_{r-1}(K)x=nx$.
Then
\[
n=\langle \Lup_{r-1}(K)x,x\rangle   \le \langle L_\Delta x,x\rangle  \le n.
\]
Therefore $L_\Delta x=nx$.
By the acyclicity of $\Delta$,
we have $\widetilde H_{r-1}(\Delta,\R)=0$, which implies that
\begin{equation}\label{homo-0}
  x\in \im\p_r^\Delta=\ker\p_{r-1}^\Delta.
\end{equation}
Also, equality in \eqref{total} gives
\[
\langle \Lup_{r-1}(K^c)x,x\rangle=0.
\]
Since $\Lup_{r-1}(K^c)$ is positive semidefinite, this is equivalent to $ (\p_r^{K^c})^*x=0$.
Observe that $\ker (\p_r^{K^c})^* \perp \im \p_r^{K^c}$.
So $x \in (\im \p_r^{K^c})^\perp$.
Because $K^c$ has the complete $(r-1)$-skeleton, we have
 $\ker \p_{r-1}^{K^c}=\ker\p_{r-1}^{\Delta}$, implying $x \in \ker \p_{r-1}^{K^c}$ by \eqref{homo-0}.
Thus,
\[
 0\ne x\in \ker \p_{r-1}^{K^c} \cap (\im \p_r^{K^c})^\perp.
\]
This gives $\widetilde H_{r-1}(K^c,\R)\ne 0$.

Conversely, suppose $\widetilde H_{r-1}(K^c,\R)\ne 0$.
Then there exists a nonzero vector $x\in \ker \p_{r-1}^{K^c} \cap (\im \p_r^{K^c})^\perp$.
As above,
\[  \ker \p_{r-1}^{K^c}=\ker\p_{r-1}^{\Delta}=\im\p_r^{\Delta}.\]
Thus $x\in\im\p_r^{\Delta}$, by Lemma \ref{lem:full-simplex} we have  $L_\Delta x=nx$.
Moreover, since $x\perp\im\p_r^{K^c}$, we have $(\p_r^{K^c})^*x=0$.
Therefore $\Lup_{r-1}(K^c)x=0$.
Using \eqref{total},
\[
\Lup_{r-1}(K)x=(L_\Delta-\Lup_{r-1}(K^c))x=nx.
\]
Thus, $n$ is an eigenvalue of $\Lup_{r-1}(K)$.
Together with the upper bound, this proves the equivalence.
\end{proof}

\begin{corollary}\label{g-eq-n}
Let $G$ be a graph on $n$ vertices.
Then
$$ \lambda_{\max}(L(G)) \le n,$$
with equality if and only if the complement $G^c$ is disconnected.
\end{corollary}

\begin{proof}
Apply Theorem \ref{eq-n} with $r=1$.
The complement $K^c$ is the graph complement $G^c$ together with the full $0$-skeleton, and
$\widetilde H_0(K^c,\R)\ne 0$ exactly when $G^c$ is disconnected.
\end{proof}

\section{A sharper upper bound}\label{section_new}

Let $K$ be an $r$-dimensional complex.
For an $(r-1)$-face $E$ of $K$, let $N_K(E)$ be the set of vertices $u$ outside $E$ such that the union of  $E$ and $\{u\}$ is an $r$-face of $K$, namely
\[
N_K(E)=\{u\in V(K)\setminus E:E\cup\{u\}\in S_r(K)\}.
\]
We call $|N_K(E)|$ the \emph{degree} of $E$ and write $\deg E=|N_K(E)|$.
Let $M_K(F)$ be the set of vertices $u\in V(K)\setminus F$ for which there exists a face $E \in \p F$ such that $E\cup\{u\}$ is an $r$-face of $K$.
Then, for every $F\in S_r(K)$, we have the following formula for the local number of $F$:
\[
 \ell_K(F)=\bigl|\bigcup_{E \in \p F} N_K(E) \bigr| =|F|+|M_K(F)|=  r+1+|M_K(F)|.
\]

Let $K$ be a finite $r$-complex with $r\ge 1$.
For each $(r+2)$-subset $H\subseteq V(K)$, let
$K[H]=\{\sigma\in K:\sigma\subseteq H\}$ be the subcomplex of $K$ induced by $H$.
We shall use
\[
\mathcal{H}_K=\{H\subseteq V(K): |H|=r+2,\ |S_r(K[H])|\ge 2\}.
\]
One may think of $\mathcal{H}_K$ as an $(r+1)$-complex built from $K$ whose $(r+1)$-faces are exactly the elements of $\mathcal{H}_K$.
Define a \emph{local mapping} $\d_K:\mathbb{R}^{B_r(K)}\to\mathbb{R}^{B_{r+1}(\mathcal{H}_K)}$ by
\[
(\d_Kz)_{[H]}=\sum_{F\in S_r(K[H])}([H]:[F])z_{[F]}.
\]
The value of $\d_Kz$ at $[H]$ is determined by the subcomplex $K[H]$.

We give a formula for the quadratic form $z^\top\Ldown_r(K)z$ by means of local number and local mapping.

\begin{lemma}\label{raylei-1}
Let $K$ be a finite $r$-complex with $r\ge 1$.
For every $z=(z_{[F]})\in\R^{B_r(K)}$,
\begin{equation}
  z^\top\Ldown_r(K)z
  = \sum_{F\in S_r(K)} \ell_K(F) z_{[F]}^2
  - \sum_{H\in\calH_K}(\d_Kz)_{[H]}^2.
\end{equation}
\end{lemma}

\begin{proof}
By definition,
\begin{equation}\label{mx-down-1}
    (\Ldown_r(K))_{[F],[G]}=
    \begin{cases}
    r+1, & F=G,\\
    ([F]:[F\cap G])([G]:[F\cap G]), & F\ne G\text{ and }F\cap G\in S_{r-1}(K),\\
     0, & \text{otherwise}.
    \end{cases}
\end{equation}

For every $H\in\mathcal{H}_K$, let $A_H$ be the signed adjacency matrix of $K[H]$ defined by
\[A_H = \Ldown_r(K[H]) - (r+1)I.\]
Let $\alpha_H\in\R^{B_r(K[H])}$ be the vector defined by $(\alpha_H)_{[F]}=([H]:[F])$.
Let $F,G \in S_r(K[H])$ and $F \cap G \in S_{r-1}(K[H])$.
Since the coefficient of $[F \cap G]$ in $\p_r\p_{r+1}[H]$ is zero, we have
\[
([H]:[F])([F]:[F \cap G])+([H]:[G])([G]:[F \cap G])=0.
\]
Thus,
\begin{equation}\label{sign-rel}
([F]:[F\cap G])([G]:[F\cap G])=-([H]:[F])([H]:[G]).
\end{equation}
By \eqref{sign-rel}, the matrix $A_H$ has the form
$$
A_H=I-\alpha_H\alpha_H^\top.
$$
Let $R_H:\R^{B_r(K)} \to \R^{B_r(K[H])}$ be the coordinate restriction map, and set $\bar{A}_H=R_H^\top A_HR_H$.
Equivalently, $\bar{A}_H$ is obtained by placing $A_H$ in the rows and columns indexed by $B_r(K[H])$ and setting all other entries equal to zero.
Thus, for every $z\in\mathbb{R}^{B_r(K)}$, one has
$$
z^\top \bar{A}_Hz=(R_Hz)^\top A_H(R_Hz).
$$

We get a decomposition of $\Ldown_r(K)$ into matrices $\bar{A}_H$:
\begin{equation}\label{down-decomp}
  \Ldown_r(K)=(r+1)I+\sum_{H\in\calH_K}\bar{A}_H.
\end{equation}
Indeed, the diagonal entries agree.
If $F\ne G$ and $|F\cap G|<r$, no $(r+2)$-set $H$ contains both $F$ and $G$, so both sides have zero $([F],[G])$-entry.
If $|F\cap G|=r$, then $H_0=F\cup G$ is the unique $(r+2)$-set containing both $F$ and $G$, and the corresponding summand contributes exactly $([F]:[F\cap G])([G]:[F\cap G])$.

Now, by \eqref{down-decomp}, we obtain
\[
\begin{aligned}
        z^\top\Ldown_r(K)z
        &= (r+1)\sum_{F\in S_r(K)}z_{[F]}^2 + \sum_{H\in\calH_K} z^\top \bar{A}_H z \\
        &= (r+1)\sum_{F\in S_r(K)}z_{[F]}^2 + \sum_{H\in\calH_K} (R_Hz)^\top A_H(R_Hz)\\
        & =(r+1)\sum_{F\in S_r(K)}z_{[F]}^2 + \sum_{H\in\calH_K}
        (R_Hz)^\top(I-\alpha_H\alpha_H^\top) (R_Hz)  \\
        &= (r+1)\sum_{F\in S_r(K)}z_{[F]}^2 + \sum_{H\in\calH_K} \sum_{F\in S_r(K[H])}z_{[F]}^2 - \sum_{H\in\calH_K}      \bigl(\sum_{F\in S_r(K[H])} \!\!\! ([H]:[F]) z_{[F]}\bigr)^2.
\end{aligned}
\]

For a fixed $F\in S_r(K)$, every $H\in\calH_K$ with $F\subseteq H$ has the form $H=F\cup\{u\}$ for a unique vertex $u\in V(K)\setminus F$.
Such an $H$ belongs to $\calH_K$ precisely when $u\in M_K(F)$.
Since
\[
 \bigl|\bigcup_{E\in\p F}N_K(E)\bigr|=r+1+|M_K(F)|= r+1 +  |\{H\in\calH_K:F\subseteq H\}|,
\]
the preceding identity becomes
\[
\begin{aligned}
z^\top\Ldown_r(K)z  & =  \sum_{F\in S_r(K)}\bigl|\bigcup_{E\in\p F}N_K(E)\bigr|z_{[F]}^2
        - \sum_{H\in\calH_K}\bigl(\sum_{F\in S_r(K[H])}([H]:[F])z_{[F]}\bigr)^2\\
       & =  \sum_{F\in S_r(K)}\ell_K(F) z_{[F]}^2
        - \sum_{H\in\calH_K} (\d_Kz)_{[H]}^2.
\end{aligned}
\]
\end{proof}

\begin{theorem}\label{sharper-bound}
Let $K$ be a finite $r$-complex with $r\ge 1$.
Then
\[
\lambda_{\max}\bigl(\Lup_{r-1}(K)\bigr)
 \le   \max_{F\in S_r(K)} \big|\bigcup_{E \in \p F} N_K(E) \big|.
\]
\end{theorem}

\begin{proof}
By \eqref{eq-ud}, it is enough to bound $\lambda_{\max}(\Ldown_r(K))$.
By Lemma \ref{raylei-1}, for every $z=(z_{[F]})\in\mathbb{R}^{B_r(K)}$,
\[
\begin{aligned}
z^\top\Ldown_r(K)z	& = \sum_{F\in S_r(K)} \ell_K(F) z_{[F]}^2
        - \sum_{H\in\calH_K}(\d_Kz)_{[H]}^2\\
	& \le \sum_{F\in S_r(K)}\ell_K(F) z_{[F]}^2 \\
    &  \le \max_{F\in S_r(K)} \ell_K(F) \sum_{F\in S_r(K)}z_{[F]}^2.
\end{aligned}
\]
Taking the maximum over all nonzero $z\in\mathbb{R}^{B_r(K)}$ gives the result.
\end{proof}

\begin{remark}
When $r=1$, $K$ is a graph $G$.
Theorem \ref{sharper-bound} implies the Das' graph bound \cite{Das2003}:
\[
 \lambda_{\max}(L(G)) \le \max_{\{u,v\} \in E(G)} |N_G(u)\cup N_G(v)|.
\]
\end{remark}

\begin{remark}
The bound in Theorem \ref{sharper-bound} is sharper than the degree-sum bound in  \cite[Theorem 3.5]{FWW2025}:
\[
\lambda_{\max}\bigl(\Lup_{r-1}(K)\bigr)
 \le \max_{F\in S_r(K)}\sum_{E\in\p F}\deg E.
\]
Indeed, for every $F\in S_r(K)$,
\[
 \bigl|\bigcup_{E\in\p F}N_K(E)\bigr|
 \le  \sum_{E\in\p F}|N_K(E)|  =  \sum_{E\in\p F}\deg E.
\]
Thus, the improvement comes from taking into account the overlaps among the
neighborhoods $N_K(E)$.
\end{remark}

Next, we characterize the equality case in the upper bound of Theorem \ref{sharper-bound}.
For an $r$-complex $K$, we denote
$ U_K=\max_{F\in S_r(K)} \ell_K(F)$.

\begin{theorem}\label{eq-vector}
Let $K$ be a finite $r$-complex with $r\ge 1$.
Then equality holds in Theorem \ref{sharper-bound}, that is,
$ \lambda_{\max}\bigl(\Lup_{r-1}(K)\bigr)=U_K$,
if and only if there exists a nonzero vector $z\in\R^{B_r(K)}$ such that $\d_Kz=0$, and $z_{[F]}=0$
for every $F\in S_r(K)$ with $\ell_K(F) <U_K$.
\end{theorem}

\begin{proof}
By Lemma \ref{raylei-1}, for every $z\in\R^{B_r(K)}$, one has
\begin{equation}\label{eq-ray-1}
 U_K\sum_{F\in S_r(K)}z_{[F]}^2 - z^\top\Ldown_r(K)z
= \sum_{F\in S_r(K)} \bigl(U_K- \ell_K(F) \bigr)z_{[F]}^2 +   \sum_{H\in\calH_K}(\d_Kz)_{[H]}^2.
\end{equation}
Both terms on the right-hand side are nonnegative.
Thus $\lambda_{\max}(\Ldown_r(K))\le U_K$, which is the upper bound proved in Theorem \ref{sharper-bound}.

Suppose first that $\lambda_{\max}(\Lup_{r-1}(K))=U_K$.
By \eqref{eq-ud}, the matrices $\Lup_{r-1}(K)$ and $\Ldown_r(K)$ have the same largest nonzero eigenvalue.
Hence $\lambda_{\max}(\Ldown_r(K))=U_K$.
Choose a nonzero eigenvector $z\in\R^{B_r(K)}$ of $\Ldown_r(K)$ corresponding to the eigenvalue $U_K$.
Then
\[
  z^\top\Ldown_r(K)z=U_K\sum_{F\in S_r(K)}z_{[F]}^2.
\]
Hence, the left-hand side of equation \eqref{eq-ray-1} equals zero.
Since both terms on the right-hand side are nonnegative, both must vanish.
The vanishing of the second term gives $\d_Kz=0$.
The vanishing of the first term gives  $z_{[F]}=0$ if $ \ell_K(F)<U_K$.
Thus $z$ has the required properties.

Conversely, suppose that there exists a nonzero vector $z\in\R^{B_r(K)}$ such that $\d_Kz=0$ and $ z_{[F]}=0$
for every $F\in S_r(K)$ with $ \ell_K(F) <U_K$.
Then both terms on the right-hand side of the displayed identity vanish.
Consequently,
\[
 z^\top\Ldown_r(K)z  =  U_K\sum_{F\in S_r(K)}z_{[F]}^2.
\]
Thus we have  $\lambda_{\max}(\Ldown_r(K))\ge U_K$.
Together with the upper bound in Theorem \ref{sharper-bound}, this yields $\lambda_{\max}(\Ldown_r(K))=U_K$.
Using \eqref{eq-ud} once more, we obtain
$\lambda_{\max}(\Lup_{r-1}(K))=U_K$.
\end{proof}

\begin{corollary}\label{new-n}
Let $K$ be a finite $r$-complex on $n$ vertices with $r\ge 1$.
Then $\lambda_{\max}\bigl(\Lup_{r-1}(K)\bigr)=n$ if and only if $U_K=n$ and there exists a nonzero vector $z\in\R^{B_r(K)}$ such that $\d_Kz=0$ and $z_{[F]}=0$ for every $F\in S_r(K)$ with $\ell_K(F) <n$.
In particular, this condition holds if and only if $\widetilde H_{r-1}(K^c,\R)\ne 0$.
\end{corollary}

\begin{proof}
By Theorem \ref{sharper-bound}, one has
\[
 \lambda_{\max}\bigl(\Lup_{r-1}(K)\bigr)\le U_K\le n.
\]
Hence $\lambda_{\max}\bigl(\Lup_{r-1}(K)\bigr)=n$ holds if and only if $U_K=n$ and equality holds in Theorem \ref{sharper-bound}.
The stated condition involving $z$ follows immediately from Theorem \ref{eq-vector}.
The final statement follows from Theorem \ref{eq-n}.
\end{proof}

\section{A family attaining the sharper bound}\label{section_fam}
The equality condition in Theorem \ref{eq-vector} gives an effective way to construct complexes that attain the upper bound in Theorem \ref{sharper-bound}.
We now describe such a family.
It is a high-dimensional analog of the extremal graph families for Das' bound \cite[Theorem 2.2]{YLT2005}.

Let
\[
V=V_0\sqcup V_1\sqcup\cdots\sqcup V_r
\]
be a partition into nonempty parts.
A pure $r$-complex $C$ on $V$ is called an \emph{$(r+1)$-partite} with respect to this partition if every $r$-face of $C$ contains exactly one vertex from each part.
For $0\le i\le r$ and for an $r$-subset $S\subseteq V\setminus V_i$ containing exactly one vertex from each part $V_j$ with $j\ne i$, write the degree of $S$ in $C$:
\[
d_i(S)=|\{v\in V_i:S\cup\{v\}\in S_r(C)\}|.
\]
We say that $C$ is \emph{semiregular} if, for each $0\le i\le r$, there exists a positive integer $d_i$ such that
$d_i(S)=d_i$ for every such set $S$.

For an $(r+1)$-partite semiregular $r$-complex $C$ on $V$, let $C^+$ be a pure $r$-complex on $V$ with $S_r(C)\subseteq S_r(C^+)$.
The faces in $S_r(C^+)\setminus S_r(C)$ will be called \emph{added faces}.
We say that $C^+$ is obtained from $C$ by \emph{admissible additions} if the following conditions hold.
First, no added face contains exactly one vertex from each part.
Second, let $F=\{x_0,\ldots,x_r\}\in S_r(C)$ with $x_i\in V_i$ for all $i$.
If an added face has the form
\[
P=(F\setminus\{x_j\})\cup\{u\}
\]
for some $u\in V_i$ and some $i\ne j$, then the subset $(F\setminus\{x_i\})\cup\{u\}$ belongs to $S_r(C)$.

We denote by $\mathcal{F}_r^+(C)$ the family of all pure $r$-complexes $C^+$ obtained from $C$ by admissible additions and satisfying
\[
\ell_{C^+}(P)=\bigl|\bigcup_{E\in\partial P}N_{C^+}(E)\bigr|\le d_0+\cdots+d_r
\]
for every added face $P\in S_r(C^+)\setminus S_r(C)$.
Finally, let $\mathcal{F}_r^+$ be the family of all pure $r$-complexes obtained in this way from some $(r+1)$-partite semiregular $r$-complex.

\begin{remark}
When $r=1$, a bipartite semiregular $1$-complex $C$ is just a bipartite semiregular graph with bipartition $V=V_0\sqcup V_1$.
For such a fixed $C$, the family $\mathcal{F}_1^+(C)$ consists exactly of the graphs obtained from $C$ by adding edges inside $V_0$ or inside $V_1$ between vertices with the same neighborhood in the other part.
Indeed, the admissibility condition says precisely that an added edge inside one part can only join two vertices with the same neighborhood in the other part.

Moreover, the extra condition on the local number in the definition of $\mathcal{F}_1^+(C)$ is automatic.
For example, if $\{u,v\}\subseteq V_0$ is an added edge, then $N_C(u)=N_C(v)\subseteq V_1$ and $|N_C(u)|=d_1$.
All vertices of $V_0$ with this same neighborhood are adjacent in $C$ to each vertex of $N_C(u)$, so there are at most $d_0$ such vertices.
Hence $\left|N_{C^+}(u)\cup N_{C^+}(v)\right|\le d_0+d_1$.
The case of an added edge inside $V_1$ is the same.
Thus, after taking all semiregular bipartite graphs $C$, the family $\mathcal{F}_1^+$ recovers the graph family associated with equality in Das' bound, which is exactly the statement of \cite[ Theorem 2.2]{YLT2005}.
\end{remark}

\begin{theorem}\label{family}
Let $C$ be a $(r+1)$-partite semiregular $r$-complex with respect to the partition $V=V_0\sqcup\cdots\sqcup V_r$.
Let $d_0,\ldots,d_r$ be the corresponding degrees.
If $C^+\in\calF_r^+(C)$, then
	\[
	\lambda_{\max}\bigl(\Lup_{r-1}(C^+)\bigr)
	=\max_{F\in S_r(C^+)}\bigl|\bigcup_{E\in\p F}N_{C^+}(E)\bigr|
	=d_0+\cdots+d_r.
	\]
In particular, every complex $K \in\calF_r^+$ attains the upper bound in Theorem \ref{sharper-bound}.
\end{theorem}

\begin{proof}
Let $F=\{x_0, \ldots, x_r\}$ be an $r$-face of $C$, where $x_i\in V_i$.
Fix $0\le i\le r$ and let $u\in V_i$.
We claim that $u\in\bigcup_{E\in\p F}N_{C^+}(E)$ if and only if
\begin{equation}\label{claim-1}
	(F\setminus\{x_i\})\cup\{u\}\in S_r(C).
\end{equation}
Indeed, if $(F\setminus\{x_i\})\cup\{u\}\in S_r(C)$, then $u\in N_{C^+}(F\setminus\{x_i\})$.
Conversely, suppose that $u\in\bigcup_{E\in\p F}N_{C^+}(E)$.
Then there exists $j$ such that $(F\setminus\{x_j\})\cup\{u\}\in S_r(C^+)$.
If $j=i$, then $(F\setminus\{x_i\})\cup\{u\}$ contains exactly one vertex from each part.
Since the added faces in $S_r(C^+)\setminus S_r(C)$ do not contain exactly one vertex from each part, this face must belong to $S_r(C)$.
If $j\ne i$, then $(F\setminus\{x_j\})\cup\{u\}$ is an added face.
By the admissible addition condition, the exchange face $(F\setminus\{x_i\})\cup\{u\}$ belongs to $S_r(C)$.

By semiregularity, for each fixed $i$, the number of vertices $u\in V_i$ such that $(F\setminus\{x_i\})\cup\{u\}\in S_r(C)$ is $d_i$.
Summing over all parts and using the claim \eqref{claim-1}, we get
\[
\ell_{C^+}(F)=\bigl|\bigcup_{E\in\p F}N_{C^+}(E)\bigr|
	= \sum_{i=0}^r \bigl|\bigcup_{E\in\p F}N_{C^+}(E) \cap V_i\bigr|
	=d_0+\cdots+d_r
\]
for every $F\in S_r(C)$.
On the other hand, by the definition of $\calF_r^+(C)$, every added face $P\in S_r(C^+)\setminus S_r(C)$ satisfies
\[	
\ell_{C^+}(P)\le d_0+\cdots+d_r.	
\]
Therefore	
\[	
U_{C^+}	=\max_{Q\in S_r(C^+)}\ell_{C^+}(Q)=	d_0+\cdots+d_r.
\]
	
By Theorem \ref{eq-vector}, it remains to construct a nonzero vector
	$z\in\R^{B_r(C^+)}$ such that $\d_{C^+}z=0$ and $z_{[Q]}=0$ for every $Q\in S_r(C^+)$ with $\ell_{C^+}(Q)<U_{C^+}$.
Orient every $r$-face of $C$ by the part order: if $F=\{x_0,\ldots,x_r\}$ with $x_i\in V_i$, then $F$ is oriented as $[x_0,\ldots,x_r]$.
Choose arbitrary orientations for the added faces.
Define $z\in\R^{B_r(C^+)}$ by
	\[ 	z_{[Q]}=
	\begin{cases}
		1,& Q\in S_r(C),\\
		0,& Q\in S_r(C^+)\setminus S_r(C).
	\end{cases}
	\]
Then $z\ne 0$.
Since $\ell_{C^+}(F)=U_{C^+}$ for every $F\in S_r(C)$, the vector $z$ satisfies the required support condition.
	
It remains to prove that $\d_{C^+}z=0$.
Fix $H\in\calH_{C^+}$.
We claim that
\begin{equation}\label{claim-2}
	(\d_{C^+}z)_{[H]} = \sum_{Q\in  S_r(C^+[H])}([H]:[Q])z_{[Q]} =0.
\end{equation}
Since $z_{[Q]}=0$ for every added face, only the faces in $S_r(C^+[H]) \cap S_r(C)$ contribute.
If $S_r(C^+[H])\cap S_r(C)=\emptyset$, then $(\d_{C^+}z)_{[H]}=0$.

Assume that $S_r(C^+[H]) \cap S_r(C)\ne\emptyset$.
Since every face of $C$ contains exactly one vertex from each part and $|H|=r+2$, the set $H$ contains two vertices in exactly one part, say
the vertices $a$ and $b$ in the part $V_i$.
Write
\[
H=\{x_0,\ldots,x_{i-1},a,b,x_{i+1},\ldots,x_r\},
\]
where $x_j\in V_j$ for every $j\ne i$.
The only $r$-subsets of $H$ which contain exactly one vertex from each part are $F_a=H\setminus\{b\}$ and $F_b=H\setminus\{a\}$.
Hence
\[
S_r(C^+[H]) \cap S_r(C)\subseteq\{F_a,F_b\}.
\]
	
We claim that $S_r(C^+[H]) \cap S_r(C)$ cannot contain exactly one of $F_a$ and $F_b$.
Suppose, for instance, that $F_a\in S_r(C)$ and $F_b\notin S_r(C)$.
Since $H\in\calH_{C^+}$, there exists another face $Q\in S_r(C^+[H])$ with $Q\ne F_a$.
The only possible faces of $C$ contained in $H$ are $F_a$ and $F_b$, so $Q$ must be an added face.
Since $Q$ and $F_a$ are distinct $(r+1)$-subsets contained in  the $(r+2)$-set $H$, we have $|Q\cap F_a|=r$.
By the admissible addition conditions, the face obtained from $F_a$ by replacing $a$ with $b$ belongs to $S_r(C)$.
This face is precisely $F_b$, a contradiction.
The same argument excludes the case $F_b\in S_r(C)$ and $F_a\notin S_r(C)$.
Therefore $S_r(C^+[H]) \cap S_r(C)$ is either empty or equal to $\{F_a,F_b\}$.

It remains to check the second case.
Let $F_{a,b}=H\setminus\{a,b\}$.
With the part-order orientations on $F_a$ and $F_b$, the vertices $a$ and $b$ occupy the same position.
Hence the two signs $([F_a]:[F_{a,b}])$ and $([F_b]:[F_{a,b}])$ are equal.
As in \eqref{sign-rel}, the boundary maps give
\[
	([H]:[F_a])([F_a]:[F_{a,b}])+([H]:[F_b])([F_b]:[F_{a,b}])=0.
\]
Thus, if $S_r(C^+[H])  \cap S_r(C)=\{F_a,F_b\}$, then
\[
	(\d_{C^+}z)_{[H]} = ([H]:[F_a])z_{[F_a]}+([H]:[F_b])z_{[F_b]} = ([H]:[F_a])+([H]:[F_b]) = 0.
\]
Hence $(\d_{C^+}z)_{[H]}=0$ for every $H\in\calH_{C^+}$.
Thus $\d_{C^+}z=0$.
	
By Theorem \ref{eq-vector}, equality holds in Theorem \ref{sharper-bound} for $C^+$.
Since $U_{C^+}=d_0+\cdots+d_r$, we obtain
\[
	\lambda_{\max}\bigl(\Lup_{r-1}(C^+)\bigr)
	=U_{C^+}=d_0+\cdots+d_r.
\]
\end{proof}

\begin{remark}
Let $r=2$, and let $V=V_0\sqcup V_1\sqcup V_2$ be a partition into three nonempty parts.
Let $C$ be the complete tripartite $2$-complex whose $2$-faces are precisely the triples containing one vertex from each part.
Then $C$ is partite semiregular with $d_0=|V_0|, d_1=|V_1|, d_2=|V_2|$.
Choose two distinct vertices $a,a'\in V_0$, a vertex $b\in V_1$, and a vertex $c\in V_2$.
Let $C_1^+$ and $C_2^+$ be defined by
\[
	S_2(C_1^+)=S_2(C)\cup\{\{a,a',b\}\}
	\text{ and }
	S_2(C_2^+)=S_2(C)\cup\{\{a,a',b\},\{a,a',c\}\}.
\]
The added faces do not contain exactly one vertex from each part.
Thus, they are genuinely higher-dimensional additions and have no analog in the graph case, where added edges are necessarily contained in a single part.
Since $C$ is complete tripartite, the admissibility condition is satisfied for both $C_1^+$ and $C_2^+$.
Moreover, for every $i=1,2$ and for every added face $P$ in either complex, one has
\[
\bigl|\bigcup_{E\in\partial P}N_{C_i^+}(E)\bigr| \le |V|=d_0+d_1+d_2.
\]
Hence $C_1^+,C_2^+\in\mathcal{F}_2^+(C)$.
By Theorem \ref{family}, both complexes attain the same  bound:
\[
\lambda_{\max}\bigl(\Lup_1(C_1^+)\bigr)=\lambda_{\max}\bigl(\Lup_1(C_2^+)\bigr)=d_0+d_1+d_2.
\]

This example shows that, in dimension two, one may add genuinely non-partite faces without increasing the largest up-Laplacian eigenvalue beyond the universal value already attained by the complete tripartite core.
\end{remark}

\begin{remark}
The converse of Theorem \ref{family} is not true in higher dimensions, even under the  path-connectedness condition on top-dimensional faces.
Let $r=2$ and let $K$ be the pure $2$-complex on the vertex set $V=\{1,2,3,4,5\}$ with
\[
S_2(K)=\{123,234,345,451,512\},
\]
where a face $\{a,b,c\}$ is abbreviated to $abc$.
One can see that $K$ is a triangulation of the M\"obius strip, and it is $2$-down path connected.
Indeed, the five $2$-faces form a tight cycle under adjacency by common edges:
\[
123\sim 234\sim 345\sim 451\sim 512\sim 123.
\]

For every $F\in S_2(K)$, one checks that
\[
\ell_K(F)=\bigl|\bigcup_{E\in\p F}N_K(E)\bigr|=5.
\]
Hence $U_K=5$.
We orient every $2$-face increasingly, and for each $H\in\mathcal{H}_K$, we also order its vertices increasingly.
Let $z\in\R^{B_2(K)}$ be defined by
\[
z_{[123]}=z_{[234]}=z_{[345]}=z_{[145]}=z_{[125]}=1.
\]
The only $4$-subsets $H$ with $|S_2(K[H])|\geq 2$ are
\[
1234,\quad 1235,\quad 1245,\quad 1345,\quad 2345.
\]
For these sets, the corresponding sums are as follows:
\[
\begin{array}{c|c|c}
H & S_2(K[H]) & \sum_{F\in S_2(K[H])}([H]:[F])z_{[F]} \\ \hline
1234 & \{123,234\}     & -z_{[123]}+z_{[234]}=0 \\[2pt]
1235 & \{123,125\}     & -z_{[123]}+z_{[125]}=0 \\[2pt]
1245 & \{125,145\}     & z_{[125]}-z_{[145]}=0 \\[2pt]
1345 & \{145,345\}     & -z_{[145]}+z_{[345]}=0 \\[2pt]
2345 & \{234,345\}     & -z_{[234]}+z_{[345]}=0.
\end{array}
\]
Thus $(\d_Kz)_{[H]}=0$ for every $H\in\calH_K$, and hence $\d_Kz=0$.
By Theorem \ref{eq-vector}, equality holds in Theorem \ref{sharper-bound}.
Therefore
\[
\lambda_{\max}\bigl(\Lup_1(K)\bigr)=U_K=5.
\]

We now show that $K\notin\calF_2^+$.
Suppose, to the contrary, that $K\in\calF_2^+(C)$ for some $3$-partite semiregular $2$-complex $C$.
Let the partition be
\[
V=V_0\sqcup V_1\sqcup V_2,
\]
and write $n_i=|V_i|$.
Since $|V|=5$, up to a permutation of the parts, their sizes are either
$3,1,1$ or $2,2,1$.

If the sizes are $3,1,1$, then semiregularity forces the three $2$-faces of $C$
to contain the same edge determined by the two singleton parts.
Since $S_2(C)\subseteq S_2(K)$, this would imply that some edge of $K$ is
contained in three $2$-faces.
However, every edge of $K$ is contained in at most two $2$-faces.
This is impossible.

If the sizes are $2,2,1$, then semiregularity forces $C$ to have all four
$2$-faces containing one vertex from each part.
In particular, all four faces contain the unique vertex in the singleton part.
Since $S_2(C)\subseteq S_2(K)$, this vertex would be contained in at least four
$2$-faces of $K$.
However, every vertex of $K$ is contained in exactly three $2$-faces.
This is also impossible.
Therefore $K\notin\calF_2^+$, although $K$ is $2$-down path connected and
attains the upper bound.
\end{remark}

\end{document}